\documentclass[a4paper,11pt]{article}
\usepackage[margin=1in]{geometry}
\usepackage{tikz}
\usepackage[small,bf,hang]{caption}

\usepackage{amsmath,amsthm,amssymb}
\usepackage{tikz}
\usepackage{hyperref}
\hypersetup{
    colorlinks=true,
    citecolor = blue,
    linkcolor=black,
    filecolor=magenta,
    urlcolor=cyan
}
  
\newtheorem{theorem}{Theorem}

\newtheorem{lemma}{Lemma}
\newtheorem{claim}{Claim}

\title{Recolouring weakly chordal graphs and the complement of triangle-free graphs}
\author{
Owen Merkel\thanks{Department of Mathematics, Wilfrid Laurier University, Waterloo, Ontario, Canada.  Email: \texttt{owenmerkel@gmail.com}.}
}

\begin{document}
\maketitle
\begin{abstract}
For a graph $G$, the $k$-recolouring graph $\mathcal{R}_k(G)$ is the graph whose vertices are the $k$-colourings of $G$ and two colourings are joined by an edge if they differ in colour on exactly one vertex. We prove that for all $n \ge 1$, there exists a $k$-colourable weakly chordal graph $G$ where $\mathcal{R}_{k+n}(G)$ is disconnected, answering an open question of Feghali and Fiala. We also show that for every $k$-colourable $3K_1$-free graph $G$, $\mathcal{R}_{k+1}(G)$ is connected with diameter at most $4|V(G)|$.
\end{abstract}

\section{Introduction}
Let $G$ be a finite simple graph with vertex-set $V(G)$ and edge-set $E(G)$. For a positive integer $k$, a \emph{$k$-colouring} of $G$ is a mapping $\alpha \colon V(G) \to \{1, 2, \ldots, k\}$ such that $\alpha(u) \neq \alpha(v)$ whenever $uv \in E(G)$. The \emph{$k$-recolouring graph}, denoted $\mathcal{R}_k(G)$, is the graph whose vertices are the $k$-colourings of $G$ and two colourings are joined by an edge if they differ in colour on exactly one vertex. We say that $G$ is \emph{$k$-mixing} if $\mathcal{R}_k(G)$ is connected. If $G$ is $k$-mixing, the \emph{$k$-recolouring diameter} of $G$ is the diameter of $\mathcal{R}_k(G)$. We say that $G$ is \emph{quadratically $k$-mixing} if the $k$-recolouring diameter of $G$ is $O(|V(G)|^2)$.

Bonamy, Johnson, Lignos, Patel, and Paulusma \cite{bonamy2014} showed that a $k$-colourable chordal or chordal bipartite graph is quadratically $(k+1)$-mixing. The authors also asked whether this statement holds more generally for perfect graphs. This was answered negatively by Bonamy and Bousquet \cite{bonamy2018} using an example of Cereceda, van den Heuvel, and Johnson \cite{cereceda2008} who showed that for all $k \ge 3$, there exists a bipartite graph that is not $k$-mixing. This started an investigation into other classes of perfect graphs which have this special property: chordal and chordal bipartite \cite{bonamy2014}, $P_4$-free \cite{bonamy2018}, distance-hereditary \cite{bonamy20142}, $P_4$-sparse \cite{biedl2020}, co--chordal, and 3-colourable ($P_5$, $\overline{P_5}$, $C_5$)-free \cite{feghali2020}.

The property of being $(k+1)$-mixing does not extend to the class of weakly chordal graphs. Feghali and Fiala \cite{feghali2020} showed that for all $k \ge 3$, there exists a $k$-colourable weakly chordal graph that is not $(k+1)$-mixing. The authors left as an open problem whether there exists an integer $l > k+1$ for which every $k$-colourable weakly chordal graph is $l$-mixing. We answer this question in the negative with the following theorem.

\begin{theorem}
\label{thm:main}
For all $n \ge 1$, there exists a $k$-colourable weakly chordal graph that is not $(k+n)$-mixing.
\end{theorem}

This question has also been investigated for the class of graphs defined by forbidding an induced path. That is, determining the values of $t$ for which a $k$-colourable $P_t$-free graph is $(k+1)$-mixing. Bonamy and Bousquet \cite{bonamy2018} showed that a $k$-colourable $P_4$-free graph is $(k+1)$-mixing, and using the same example of Cereceda, van den Heuvel, and Johnson \cite{cereceda2008}, showed that for all $k \ge 3$ and $t \ge 6$, there is a $k$-colourable $P_t$-free graph that is not $(k+1)$-mixing. It was also mistakenly reported in \cite{bonamy2018} that there exists a $k$-colourable $P_5$-free graph that is not $(k+1)$-mixing (see \cite{merkel2020}). This leaves $t=5$ as the last open case.

In this paper, we investigate this question for a subclass of $P_5$-free graphs, namely $3K_1$-free graphs. This class of graphs also includes the perfect class of co--bipartite graphs.

\begin{theorem}
\label{thm:3k1}
If $G$ is a $k$-colourable $3K_1$-free graph, then $G$ is $(k+1)$-mixing and the $(k+1)$-recolouring diameter of $G$ is at most $4|V(G)|$.
\end{theorem}

The proof of Theorem \ref{thm:3k1} leads to a polynomial time algorithm to find a path of length at most $4|V(G)|$ between any two $(k+1)$-colourings of $G$ in the recolouring graph.

The rest of the paper is organized as follows. In Section \ref{sec:pre} we give definitions and notation used throughout the paper. We prove Theorem \ref{thm:main} in Section \ref{sec:wc} and we prove Theorem \ref{thm:3k1} and in Section \ref{sec:3k1}. We end with some discussion on future work in Section \ref{sec:conc}.

\section{Preliminaries}
\label{sec:pre}

For a graph $G$, a \emph{clique} of $G$ is a set of pairwise adjacent vertices and a \emph{stable set} is a set of pairwise non-adjacent vertices. A graph $G$ is \emph{$3K_1$-free} if the maximum number of vertices in a stable set of $G$ is at most 2. The \emph{clique number} of $G$, denoted by $\omega(G)$, is the maximum number of vertices in a clique of $G$. The \emph{chromatic number} of $G$, denoted by $\chi(G)$, is the minimum $k$ such that $G$ is $k$-colourable. Clearly, $\chi(G) \ge \omega(G)$. A graph $G$ is perfect if for all induced subgraphs $H$ of $G$, $\chi(H) = \omega(H)$.

The \emph{complement} of $G$, denoted $\overline{G}$, is the graph with vertex-set $V(G)$ such that $uv \in E(\overline{G})$  exactly when $uv \notin E(G)$. A graph is \emph{bipartite} if its vertices can be partitioned into two stable sets and a graph is \emph{co--bipartite} if it is the complement of a bipartite graph. A \emph{hole} is a chordless cycle on at least five vertices and an \emph{antihole} is the complement of a hole. A hole is \emph{even} or \emph{odd} if it has an even or odd number of vertices, respectively. For a set of graphs $\mathcal{H}$, we say that $G$ is $\mathcal{H}$-free if $G$ does not contain an induced subgraph isomorphic to any graph in $\mathcal{H}$. A graph is perfect if and only if it is (odd hole, odd antihole)-free \cite{SPGT}. A graph is \emph{weakly chordal} if it is (hole, antihole)-free. Clearly, a graph $G$ is weakly chordal if and only if $\overline{G}$ is weakly chordal.

For a vertex $v \in V(G)$, the \emph{open neighbourhood} of $v$ is the set of vertices adjacent to $v$ in $G$. The \emph{closed neighbourhood of $v$} is the set of vertices adjacent to $v$ in $G$ together with $v$. For $X,Y \subseteq V(G)$, we say that $X$ is \emph{complete} to $Y$ if every vertex in $X$ is adjacent to every vertex in $Y$. If no vertex of $X$ is adjacent to a vertex of $Y$, we say that $X$ is \emph{anticomplete} to $Y$. Let $G$ and $H$ be vertex-disjoint graphs and let $v \in V(G)$. By \emph{substituting} $H$ for the vertex $v$ of $G$, we mean taking the graph $G-v$ and adding an edge between every vertex of $H$ and every vertex of $G-v$ that is adjacent to $v$ in $G$.

For a colouring $\alpha$ of $G$ and $X \subseteq V(G)$, we say that the colour $c$ \emph{appears} in $X$ if $\alpha(x) = c$ for some $x \in X$.
A $k$-colouring of a graph $G$ is called \emph{frozen} if it is an isolated vertex in the recolouring graph $\mathcal{R}_k(G)$. In other words, for every vertex $v \in V(G)$, each of the $k$ colours appears in the closed neighbourhood of $v$.

\section{Frozen colourings of weakly chordal graphs}
\label{sec:wc}

\begin{figure}
\centering
\begin{tikzpicture}[scale=0.35]
\tikzstyle{vertex}=[circle, draw, fill=black, inner sep=0pt, minimum size=5pt]

    \node[vertex, label=left:2](1) at (0,0) {};
    \node[vertex, label=above:1](2) at (5,5) {};
    \node[vertex, label=below:3](3) at (5,2.5) {};
	\node[vertex, label=right:1](4) at (2.5,0) {};
	\node[vertex, label=above:3](5) at (5,-2.5) {};
	\node[vertex, label=below:1](6) at (5,-5) {};
	\node[vertex, label=left:1](7) at (7.5,0) {};
	\node[vertex, label=right:2](8) at (10,0) {};
	
	\draw(1)--(2);
    \draw(1)--(3);
    \draw(1)--(4);
    \draw(1)--(5);
    \draw(1)--(6);
    \draw(2)--(3);
    \draw(2)--(8);
    \draw(3)--(4);
    \draw(3)--(7);
    \draw(3)--(8);
    \draw(4)--(5);
    \draw(5)--(6);
    \draw(5)--(7);
    \draw(5)--(8);
    \draw(6)--(8);
    \draw(7)--(8);
    
\end{tikzpicture}
\hspace{5mm}
\begin{tikzpicture}[scale=0.35]
\tikzstyle{vertex}=[circle, draw, fill=black, inner sep=0pt, minimum size=5pt]

    \node[vertex, label=left:1](1) at (0,0) {};
    \node[vertex, label=above:4](2) at (5,5) {};
    \node[vertex, label=below:2](3) at (5,2.5) {};
	\node[vertex, label=right:3](4) at (2.5,0) {};
	\node[vertex, label=above:4](5) at (5,-2.5) {};
	\node[vertex, label=below:2](6) at (5,-5) {};
	\node[vertex, label=left:1](7) at (7.5,0) {};
	\node[vertex, label=right:3](8) at (10,0) {};
	
	\draw(1)--(2);
    \draw(1)--(3);
    \draw(1)--(4);
    \draw(1)--(5);
    \draw(1)--(6);
    \draw(2)--(3);
    \draw(2)--(8);
    \draw(3)--(4);
    \draw(3)--(7);
    \draw(3)--(8);
    \draw(4)--(5);
    \draw(5)--(6);
    \draw(5)--(7);
    \draw(5)--(8);
    \draw(6)--(8);
    \draw(7)--(8);
    
\end{tikzpicture}
\caption{A 3-colouring and frozen 4-colouring of $G_1$.}
\label{fig:g1}
\end{figure}
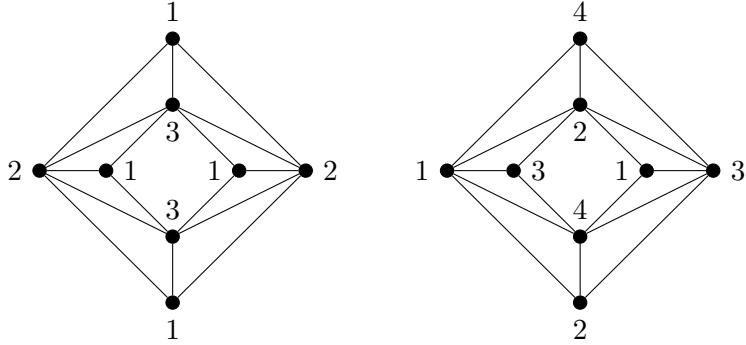

In this section we prove Theorem \ref{thm:main}. One technique to prove that a graph $G$ is not $k$-mixing is to exhibit a frozen $k$-colouring of $G$. We construct a family of graphs $\{G_n \mid n \ge 1\}$ such that $G_n$ is a $k$-colourable weakly chordal graph that has a frozen $(k+n)$-colouring. See Figure \ref{fig:g1} for a 3-colouring and a frozen 4-colouring of $G_1$.  For $n \ge 2$, we recursively construct $G_n$ by substituting $G_{n-1}$ into four vertices of $G_1$ (see Figure \ref{fig:construction}).

We first prove that substituting a weakly chordal graph for some vertex of a weakly chordal graph results in a weakly chordal graph. We note that there might be a proof of this in the literature, and for example, Lov\'{a}sz proved an analogous theorem for perfect graphs \cite{lovasz1972}.

\begin{theorem}
\label{thm:subs}
Substituting a weakly chordal graph for some vertex of a weakly chordal graph results in a weakly chordal graph.
\end{theorem}
\begin{proof}
Let $G_1$ and $G_2$ be vertex-disjoint weakly chordal graphs and let $v \in V(G_1)$. Let $G$ be the graph obtained by substituting $G_2$ for the vertex $v$ of $G_1$.

By contradiction, suppose $G$ contains a hole $H$. Then $H$ must contain at least 2 vertices $v_1, v_2$ of $G_2$ since $G_1$ is a weakly chordal graph. Furthermore, since $G_2$ is a weakly chordal graph, $H$ must contain at least one vertex $x$ in $G_1$ that is either adjacent to $v_1$ or $v_2$ in $G$. But any vertex of $G - G_2$ that has a neighbour in $G_2$ is complete to $G_2$. So $x$ must be adjacent to both $v_1$ and $v_2$. Since $x$ can have at most two neighbours in $H$ and since $H$ is a hole, $H$ cannot contain any more neighbours of $x$. Then $H$ cannot contain another vertex from $G_2$ since $x$ is complete to $G_2$. But any other vertex of $H$ adjacent to $v_1$ or $v_2$ must be adjacent to both $v_1$ and $v_2$, so $H$ cannot be a hole, a contradiction.

Now suppose that $G$ contains an antihole. Note that $\overline{G}$ is obtained by substituting the weakly chordal graph $\overline{G_2}$ into the vertex $v$ of the weakly chordal graph $\overline{G_1}$. But since $G$ contains an antihole, $\overline{G}$ contains a hole, a contradiction.
\end{proof}

\begin{lemma}
\label{lem:wc}
For all $n \ge 1$, $G_n$ is a weakly chordal graph.
\end{lemma}
\begin{proof}
The proof is by induction on $n$. It is easy to verify that $G_1$ is weakly chordal and so the statement holds for $n=1$. By the induction hypothesis, $G_{n-1}$ is a weakly chordal graph. The graph $G_n$ is constructed by substituting $G_{n-1}$ into 4 vertices of $G_1$. Since $G_1$ and $G_{n-1}$ are both weakly chordal graphs, it follows from Theorem \ref{thm:subs} that $G_n$ is a weakly chordal graph.
\end{proof}

We are now ready to prove Theorem \ref{thm:main}, which follows from Lemma \ref{lem:col} and \ref{lem:froz}. Recalling the notation used in Figure \ref{fig:construction}, note that in $G_n$ and for $v \in \{w,x,y,z\}$, $v$ is complete to exactly three copies of $G_{n-1}$ and anticomplete to the other copy of $G_{n-1}$. For $v \in \{w,x,y,z\}$, let $G_{n-1}^v$ denote the copy of $G_{n-1}$ in $G_n$ that is anticomplete to $v$.

\begin{lemma}
\label{lem:col}
For all $n \ge 1$, $\chi(G_n) = \omega(G_n)=2n+1$.  
\end{lemma}
\begin{proof}
The proof is by induction on $n$. The statement holds for $n=1$ since $G_1$ is 3-colourable and contains a clique of size 3 (see Figure \ref{fig:g1}). By the induction hypothesis, $\chi(G_{n-1})=\omega(G_{n-1})=2n-1$. Fix a $(2n-1)$-colouring $\alpha$ of $G_{n-1}$. We show how to extend $\alpha$ to a $(2n+1)$-colouring of $G_n$. Since each copy of $G_{n-1}$ is pairwise anticomplete, we can colour each copy of $G_{n-1}$ identically using $\alpha$. To complete this colouring of $G_n$, we make $\alpha(w)=\alpha(z)=2n$ and $\alpha(x)=\alpha(y)=2n+1$. Since $wz, xy \notin E(G)$, this gives a proper $(2n+1)$-colouring of $G_n$. To find a clique of size $2n+1$ in $G_n$, take a clique $K$ of size $2n-1$ in $G_{n-1}^z$. Then since $wx \in E(G)$ and since $\{w,x\}$ is complete to $G_{n-1}^z$, it follows that $K \cup \{w,x\}$ is a clique of size $2n+1$ in $G_n$.
\end{proof}

\begin{lemma}
\label{lem:froz}
For all $n \ge 1$, $G_n$ has a frozen $(3n+1)$-colouring.
\end{lemma}
\begin{proof}
The proof is by induction on $n$. The statement holds for $n=1$ since $G_1$ has a frozen 4-colouring (see Figure \ref{fig:g1}). By the induction hypothesis, $G_{n-1}$ has a frozen $(3n-2)$-colouring. To construct a frozen $(3n+1)$-colouring $\alpha$ of $G_n$, we take a frozen $(3n-2)$-colouring of each copy of $G_{n-1}$ in $G_n$ using a different set of colours. 

For $v \in \{w,x,y,z\}$, let $\alpha^v$ denote the colouring of $G_n$ restricted to the subgraph $G_{n-1}^v$. Let $\alpha^w$ be a frozen $(3n-2)$-colouring of $G_{n-1}^w$ using the colours $\{1, 2, \ldots, 3n-2\}$. Let $\alpha^x$, $\alpha^y$, $\alpha^z$ be frozen $(3n-2)$-colourings of $G_{n-1}^x$, $G_{n-1}^y$, $G_{n-1}^z$ using the colours $\{1, 2, \ldots, 3n-3, 3n-1\}$, $\{1, 2, \ldots, 3n-3, 3n\}$, $\{1, 2, \ldots, 3n-3, 3n+1\}$, respectively. Since each each copy of $G_{n-1}$ is pairwise anticomplete, this creates no conflicts. To complete this colouring of $G_n$, make $\alpha(w)=3n-2$, $\alpha(x)=3n-1$, $\alpha(y)=3n$, and $\alpha(z)=3n+1$. Note that for each $v \in \{w, x, y, z\}$, $\alpha(v)$ only appears on $v$ and in $G_{n-1}^v$. Since $v$ is anticomplete to $G_{n-1}^v$, this creates no conflicts. Therefore, $\alpha$ is a proper $(3n+1)$-colouring of $G_n$. 

To see that $\alpha$ is a frozen colouring, first examine a vertex $u$ in $G_{n-1}^v$ for $v \in \{w,x,y,z\}$. By construction, there are $3n-2$ colours appearing on the closed neighbourhood of $u$ in $G_{n-1}^v$. Also by construction, the remaining 3 colours are used to colour $\{w,x,y,z\} \setminus \{v\}$. Since each of $\{w,x,y,z\} \setminus \{v\}$ is complete to $G_{n-1}^v$, all $3n+1$ colours appear on the closed neighbourhood of $u$ and it cannot be recoloured. Now examine vertex $v \in \{w,x,y,z\}$. Since $v$ is complete to each $G_{n-1}^u$ for $u \in \{w,x,y,z\} \setminus \{v\}$, there are $3n$ colours appearing on the open neighbourhood of $v$. Since $\alpha$ is a proper colouring, the last colour is being used to colour $v$ and so it cannot be recoloured.
\end{proof}

\begin{figure}
\centering
\begin{tikzpicture}[scale=0.12]
\tikzstyle{vertex}=[circle, draw, fill=black, inner sep=0pt, minimum size=5pt]
\tikzstyle{node}=[circle, draw, fill=black, inner sep=0pt, minimum size=8pt]
\tikzstyle{blowup}=[circle, draw, fill=white, inner sep=0pt, minimum size=40pt]
    
    \draw[line width=2] (0, 0) ellipse (20cm and 20cm);
    \draw[line width=2] (0, 0) ellipse (10cm and 10cm);
    \draw (0, 0) ellipse (20cm and 10cm);
    
    \node[node, label=left:$w$](40) at (-20,0) {};
    \node[node, label=below:$x$](41) at (0,10) {};
    \node[node, label=above:$y$](42) at (0,-10) {};
    \node[node, label=right:$z$](43) at (20,0) {};
    
    \node[blowup](44) at (0,20){$G_{n-1}^y$};
    \node[blowup](45) at (-10,0){$G_{n-1}^z$};
    \node[blowup](46) at (10,0){$G_{n-1}^w$};
    \node[blowup](47) at (0,-20){$G_{n-1}^x$};
    
    \draw[line width=2](40)--(45);
    \draw[line width=2](41)--(44);
    \draw[line width=2](42)--(47);
    \draw[line width=2](43)--(46);
    
\end{tikzpicture}
\caption{The graph $G_n$. A bold line indicates that all possible edges are present.}
\label{fig:construction}
\end{figure}
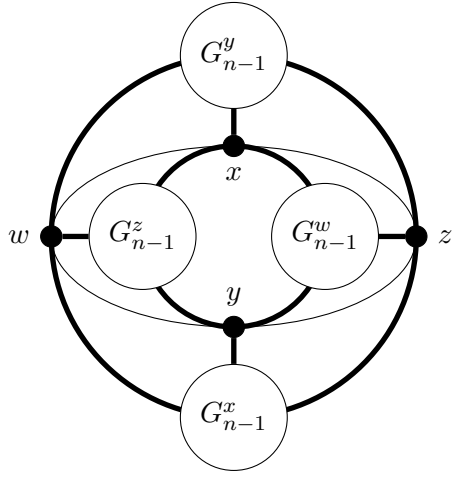

\section{Recolouring the complement of triangle-free graphs}
\label{sec:3k1}

In this section we prove Theorem \ref{thm:3k1}. Note that in any colouring of a $3K_1$-free graph at most two vertices share the same colour. With this in mind, it is not hard to see that an optimal colouring of a $3K_1$-free graph can be found in polynomial time by finding a maximum matching in the complement. We begin by proving the following lemma.

\begin{lemma}
\label{lem:useful}
Let $G$ be a $k$-colourable $3K_1$-free graph. In any $(k+1)$-colouring of $G$, there exists a colour $c$ that either does not appear in $G$ or is used to colour exactly one vertex of $G$.
\end{lemma}
\begin{proof}
Let $G$ be as in the statement of the lemma and fix some $(k+1)$-colouring of $G$. We can assume that all $k+1$ colours appear on the vertices of $G$ since, if not, the first condition is satisfied. Now by contradiction assume that all $k+1$ colours appear twice on the vertices of $G$. We know that $|V(G)| \le 2\chi(G)$ since we can partition the vertices of $G$ into at most $\chi(G)$ stable sets, each having at most two vertices. But since all $k+1$ colours appear twice on the vertices of $G$, we have $|V(G)| \ge 2(k+1) > 2\chi(G)$, a contradiction.
\end{proof}

Let $\gamma$ be a $\chi(G)$-colouring of $G$ and let $\mathcal{C}$ be the partition of the vertices of $G$ given by the colour classes of $\gamma$. Given two colourings $\alpha$ and $\beta$ of $G$, our strategy is to first recolour each to a $\chi(G)$-colouring $\alpha'$ and $\beta'$ whose colour classes correspond exactly to the partition $\mathcal{C}$, and then use the following Renaming Lemma. 

\begin{lemma}[Renaming Lemma \cite{bonamy2018}]
\label{lem:recolour}
If $\alpha'$ and $\beta'$ are two $k$-colourings of $G$ that induce the same partition of vertices into colour classes, then $\alpha'$ can be recoloured into $\beta'$ in $\mathcal{R}_{k+1}(G)$ by recolouring each vertex at most 2 times.
\end{lemma}

\begin{proof}[Proof of Theorem \ref{thm:3k1}]
Let $G$ be a $k$-colourable $3K_1$-free graph and let $\alpha$ and $\beta$ be two $(k+1)$-colourings of $G$. Fix a $\chi(G)$-colouring $\gamma$ of $G$ and let $\mathcal{C}$ be the partition of $V(G)$ given by the colour classes of $\gamma$. Note that $|\mathcal{C}|=\chi(G)$ and each colour class $C \in \mathcal{C}$ has one or two vertices. 

\begin{claim}
The colouring $\alpha$ can be recoloured into a $\chi(G)$-colouring $\alpha'$ of $G$ such that $\alpha'$ and $\gamma$ partition the vertices of $G$ into the same colour classes by recolouring each vertex at most once.
\end{claim}

We prove the claim by induction on $\chi(G)$. For $\chi(G)=1$ the claim is trivial. Now assume the statement holds for $\chi(G)-1$. By Lemma \ref{lem:useful}, there exists some colour $c$ of $\alpha$ that either does not appear in $G$ or appears on exactly one vertex of $G$.

First suppose the colour $c$ appears in $G$ and let $u$ be the vertex coloured $c$. Let $C$ be the colour class of $\gamma$ which contains $u$. If $C$ contains some other vertex $v$ then, from $\alpha$, recolour $v$ with $c$. If instead $c$ does not appear in $G$, we select $u$, $v$, and $C$ as follows. Take some colour class of $\alpha$ that is not a colour class of $\gamma$ (if no such colour class exists we are done) and some vertex $u$ in this colour class. From $\alpha$, recolour $u$ with the colour $c$. Let $C$ be the colour class of $\gamma$ which contains $u$. If there is another vertex $v \in C$ then recolour $v$ to the colour $c$. This can be done since $uv \notin E$ and no other vertex is coloured $c$. 

Let $\alpha_C$ be the current colouring of $G$ restricted to $G - C$ with $c$ taken out of its set of colours. Let $\gamma_C$ be the colouring $\gamma$ restricted to $G - C$.

Since $\gamma$ is a $\chi(G)$-colouring of $G$, it follows that $\chi(G - C) = \chi(G)-1$. Then $\alpha_C$ is a $k$-colouring of $G-C$ (since we removed the colour $c$) and $k \ge \chi(G-C) + 1$. By the induction hypothesis, $\alpha_C$ can be recoloured into a $(\chi(G)-1)$-colouring $\alpha_C'$ of $G-C$ such that $\alpha_C'$ and $\gamma_C$ partition the vertices of $G$ into the same colour classes by recolouring each vertex at most once. Since the colour of $u$ and $v$ are never used again, this recolouring sequence from $\alpha_c$ to $\alpha_c'$ can be extended to a recolouring sequence between $\alpha$ and $\alpha'$. Since $u$ and $v$ are recoloured at most once, each vertex of $G$ is recoloured at most once. This completes the proof of the claim.

Similarly, $\beta$ can be recoloured into a $\chi(G)$-colouring $\beta'$ such that $\beta'$ and $\gamma$ partition the vertices of $G$ into the same colour classes by recolouring each vertex at most once. By Lemma \ref{lem:recolour}, we can recolour $\alpha'$ into $\beta'$ by recolouring each vertex at most twice. This gives us a recolouring sequence from $\alpha$ to $\beta$ by recolouring each vertex at most 4 times.
\end{proof}


\section{Conclusion}
\label{sec:conc}
In this paper, we answered an open question of Feghali and Fiala by showing that for all $n \ge 1$, there exists a $k$-colourable weakly chordal graph with a frozen $(k+n)$-colouring. We also showed that every $k$-colourable $3K_1$-free graph is $(k+1)$-mixing with a linear $(k+1)$-recolouring diameter. It is an open problem whether a $k$-colourable $P_5$-free graph is $(k+1)$-mixing \cite{merkel2020}. This question has been answered for several subclasses of $P_5$-free graphs. These include when $k=2$ \cite{bonamy2014}, for co--chordal graphs, for ($P_5$, $\overline{P_5}$, $C_5$)-free graphs and $k=3$ \cite{feghali2020}, for $P_4$-sparse graphs \cite{biedl2020}, and now for $3K_1$-free graphs. It may be hard to answer this question for the entire class of $P_5$-free graphs and so it would be interesting to continue studying subclasses of $P_5$-free graphs for which this question can be answered.

\section*{Acknowledgements}
The author thanks Carl Feghali for comments and discussion that greatly improved the paper. The author was partially supported by Natural Sciences and Engineering Research Council of Canada (NSERC) grant RGPIN-2016-06517.


\begin{thebibliography}{10}

\bibitem{biedl2020}
T.~Biedl, A.~Lubiw, O.~Merkel.
\newblock Building a larger class of graphs for efficient reconfiguration of vertex colouring.
\newblock arXiv:2003.01818 [cs.DM], 2020.

\bibitem{bonamy20142}
M.~Bonamy, N.~Bousquet.
\newblock Recoloring graphs via tree decompositions.
\newblock arXiv:1403.6386 [cs.DM], 2014.

\bibitem{bonamy2018}
M.~Bonamy, N.~Bousquet.
\newblock Recoloring graphs via tree decompositions.
\newblock {\em European Journal of Combinatorics}, 69:200--213, 2018.

\bibitem{bonamy2014}
M.~Bonamy, M.~Johnson, I.~Lignos, V.~Patel, D.~Paulusma.
\newblock Reconfiguration graphs for vertex colourings of chordal and chordal bipartite graphs.
\newblock {\em Journal of Combinatorial Optimization}, 27:132--143, 2014.

\bibitem{cereceda2008}
L.~Cereceda, J.~van den Heuvel, M.~Johnson.
\newblock Connectedness of the graph of vertex-colourings.
\newblock{\em Discrete Mathematics}, 308:913--919, 2008.

\bibitem{SPGT}
M.~Chudnovsky, N.~Robertson, P.~Seymour, and R.~Thomas.
\newblock The strong perfect graph theorem.
\newblock {\em Annals of Mathematics}, 164:51--229, 2006.

\bibitem{feghali2020}
C.~Feghali, J.~Fiala.
\newblock Reconfiguration graph for vertex colourings of weakly chordal graphs.
\newblock{\em Discrete Mathematics}, 343:111733, 2020.

\bibitem{lovasz1972}
L.~Lov\'{a}sz.
\newblock Normal hypergraphs and the perfect graph conjecture.
\newblock{\em Discrete Mathematics}, 2:253--267, 1972.

\bibitem{merkel2020}
O.~Merkel.
\newblock Building a larger class of graphs for efficient reconfiguration of vertex colouring.
\newblock Master's thesis, Univeristy of Waterloo, 2020. http://hdl.handle.net/10012/15842

\end{thebibliography}
\end{document}